\documentclass[12pt]{article}
\usepackage{amssymb}
\usepackage{amsthm}
\usepackage{amsmath}
\usepackage{color}
\usepackage[T1]{fontenc}

\newtheorem{Definition}{Definition}%[section]
\newtheorem{Lemma}[Definition]{Lemma}

\newtheorem{Theorem}[Definition]{Theorem}

\newtheorem{Example}[Definition]{Example}
\newtheorem{Corollary}[Definition]{Corollary}
\title{Residuation in non-associative MV-algebras\thanks{This is a non-refereed version of a paper which will be published by De Gruyter in Mathematica Slovaca.}}
\author{Ivan~Chajda and Helmut~L\"anger}
\date{}
\begin{document}
\footnotetext[1]{Support of the research by the bilateral project entitled "New perspectives on residuated posets", supported by the Austrian Science Fund (FWF), project I~1923-N25, and the Czech Science Foundation (GA\v CR), project 15-34697L, and by \"OAD, project CZ~04/2017, as well as support of the first author by IGA, project P\v rF 2017012, is gratefully acknowledged.}
\maketitle
\begin{abstract}
It is well known that every MV-algebra can be converted into a residuated lattice satisfying divisibility and the double negation law. In our previous papers we introduced the concept of an NMV-algebra which is a non-associative modification of an MV-algebra. The natural question arises if an NMV-algebra can be converted into a residuated structure, too. Contrary to MV-algebras, NMV-algebras are not based on lattices but only on directed posets and the binary operation need not be associative and hence we cannot expect to obtain a residuated lattice but only an essentially weaker structure called a conditionally residuated poset. Considering several additional natural conditions we show that every NMV-algebra can be converted in such a structure. Also conversely, every such structure can be organized into an NMV-algebra. Further, we study a bit more stronger version of an algebra where the binary operation is even monotonous. We show that such an algebra can be organized into a residuated poset and, conversely, every residuated poset can be converted in this structure.
\end{abstract}

{\bf AMS Subject Classification:} 06D35, 03G10, 06A11

{\bf Keywords:} MV-algebra, non-associative MV-algebra, conditional adjointness, residuated poset, directoid

The concept of an MV-algebra was introduced by C.~C.~Chang (\cite C) as an algebraic semantic for the \L ukasiewicz many-valued logic. His definition is rather complicated but it was simplified by Cignoli, D'Ottaviano and Mundici (\cite{CDM}). In some applications, e.g.\ in expert systems, it seems that associativity of the binary operation $\oplus$, which models disjunction, cannot be accepted. Hence, a non-associative version of an MV-algebra, the so-called NMV-algebra, was introduced by the first author and J.~K\"uhr in \cite{CK} and later studied by the authors in \cite{CL}.

It is well-known that MV-algebras can be considered as so-called residuated lattices. In fact, an (integral) residuated lattice is an MV-algebra if and only if it satisfies the so-called divisibility and the double negation law, see e.g.\ \cite B for details. This motivated us to investigate under what conditions a similar characterization is possible also for NMV-algebras.

It is worth noticing that the adjointness property in residuated structures strongly depends on the associativity of the binary multiplication. Hence, we cannot expect that residuation will hold in full sense also in the non-associative case. On the other hand, using certain reasonable restriction, adjointness can be modified for NMV-algebras provided they satisfy some more condition.

At first, we repeat the definition of an NMV-algebra from \cite{CK}.

\begin{Definition}\label{def1}
A {\em non-associative MV-algebra} {\rm(}{\em NMV-algebra}{\rm)} is an algebra ${\bf A}=(A,\oplus,$ $\neg,0)$ of type $(2,1,0)$ satisfying the identities
\begin{eqnarray*}
                                                             x\oplus y & \approx & y\oplus x, \\
                                                              x\oplus0 & \approx & x, \\
                                                          \neg(\neg x) & \approx & x, \\
                                                              x\oplus1 & \approx & 1, \\
                                          \neg(\neg x\oplus y)\oplus y & \approx & \neg(\neg y\oplus x)\oplus x,\label{equ3} \\
\neg x\oplus(\neg(\neg(\neg(\neg x\oplus y)\oplus y)\oplus z)\oplus z) & \approx & 1, \\
                                               \neg x\oplus(x\oplus y) & \approx & 1.
\end{eqnarray*}
Here and in the following $1$ is an abbreviation for $\neg0$. The fifth identity is usually called the {\em\L ukasiewicz axiom}.
\end{Definition}

In any NMV-algebra, we can define the following term operations:
\begin{eqnarray*}
x\rightarrow y & := & \neg x\oplus y, \\
     x\sqcup y & := & (x\rightarrow y)\rightarrow y, \\
    x\otimes y & := & \neg(\neg x\oplus\neg y), \\
     x\sqcap y & := & \neg(\neg x\sqcup\neg y), \\
		       x^y & := & x\rightarrow y
\end{eqnarray*}
and the following binary relation
\[
x\leq y\text{ if and only if }x\rightarrow y=1.
\]
This relation will be called the {\em induced order} of the NMV-algebra. The following identities for NMV-algebras are immediate or follow from \cite{CK} or \cite{CL}:
\begin{eqnarray*}
((x\rightarrow y)\rightarrow y)\rightarrow y & \approx & x\rightarrow y, \\
                              x\rightarrow y & \approx & (x\sqcup y)^y, \\
                              x\rightarrow y & \approx & \neg y\rightarrow\neg x, \\
                                  x\otimes y & \approx & \neg(x\sqcup\neg y)^{\neg y}, \\
               (x\rightarrow y)\rightarrow y & \approx & (y\rightarrow x)\rightarrow x, \\
                              1\rightarrow x & \approx & x, \\
                               x\rightarrow0 & \approx & \neg x, \\
                    x\otimes(x\rightarrow y) & \approx & x\sqcap y, \\
                x\rightarrow(y\rightarrow x) & \approx & 1, \\
								    (x\sqcup y)\rightarrow y & \approx & x\rightarrow y.
\end{eqnarray*}
From the third identity we conclude
\[
x\leq y\text{ if and only if }\neg y\leq\neg x.
\]
The fifth identity is in fact the \L ukasiewicz axiom.

We recall several terms and concepts which will be used throughout the paper.

Let $(A,\leq)$ be a poset. For $x,y\in A$ we denote by $U(x,y)$ the set $\{z\in A\mid x,y\leq z\}$  . A {\em poset} $(A,\leq)$ is called {\em directed} if $U(x,y)\neq\emptyset$ for all $x,y\in A$. Of course, every poset $(A,\leq)$ with greatest element $1$ is directed since $1\in U(x,y)$ for all $x,y\in A$.

A groupoid $(A,\sqcup)$ is called a {\em directoid} (or {\em commutative directoid} in \cite{JQ}) if it satisfies the identities $x\sqcup x\approx x$, $x\sqcup y\approx y\sqcup x$ and $x\sqcup((x\sqcup y)\sqcup z)\approx(x\sqcup y)\sqcup z$ (see e.g.\ \cite{CL11} for details). It was shown by Je\v zek and Quackenbush (\cite{JQ}) that every directed poset $(A,\leq)$ can be converted into a directoid by defining $x\sqcup y:=\max(x,y)$ if $x$ and $y$ are comparable and $x\sqcup y=y\sqcup x\in U(x,y)$ otherwise ($x,y\in A$); the choice of $x\sqcup y\in U(x,y)$ is arbitrary but fixed. Then $x\leq y$ is equivalent to $x\sqcup y=y$. Also conversely, if $(A,\sqcup)$ is a directoid and we define $x\leq y$ if and only if $x\sqcup y=y$ then $(A,\leq)$ is a directed poset.

Let $(A,\leq,1)$ be a poset with greatest element $1$. For $a\in A$, the interval $[a,1]$ will be called a {\em section}. We say that $(A,\leq,1)$ is a {\em poset with switching section involutions} if for every $a\in A$ there exists a mapping $x\mapsto x^a$ of $[a,1]$ into itself such that $a^a=1$, $1^a=a$ and $(x^a)^a=x$ for all $x\in A$. We say that this poset is a {\em poset with section antitone involutions} (shortly, SAI) if, moreover, $x,y\in[a,1]$ and $x\leq y$ together imply $y^a\leq x^a$. Hence every section antitone involution is a switching one.

The following is immediate or follows from \cite{CK} or \cite{CL}:

\begin{Lemma}
Let $(A,\oplus,\neg,0)$ be an {\rm NMV}-algebra, $a,b\in A$ and $\leq$ its induced order. Then the following hold:
\begin{eqnarray*}
& & (A,\leq,1)\text{ is a poset with greatest element }1, \\
& & (A,\sqcup)\text{ is a commutative directoid}, \\
& & ^a|[a,1]\text{ is a switching involution on }([a,1],\leq), \\
& & a,b\leq a\oplus b, \\
& & a,b\leq a\sqcup b, \\
& & a\sqcap b\leq a,b, \\
& & \neg a\leq a\rightarrow b.
\end{eqnarray*}
\end{Lemma}

The following example serves as an inspiration for our investigations concerning NMV-algebras.

\begin{Example}\label{ex1}
Let $A=\{0,a,b,c,e,1\}$ and define the operations $\oplus$ and $\neg$ as follows:
\[
\begin{array}{c|cccccc}
\oplus & 0 & a & b & c & d & 1 \\
\hline
  0    & 0 & a & b & c & d & 1 \\
	a    & a & d & c & c & 1 & 1 \\
	b    & b & c & d & 1 & d & 1 \\
	c    & c & c & 1 & 1 & 1 & 1 \\
	e    & d & 1 & d & 1 & 1 & 1 \\
	1    & 1 & 1 & 1 & 1 & 1 & 1
\end{array}
\quad
\begin{array}{c|cccccc}
   x   & 0 & a & b & c & d & 1 \\
\hline
\neg x & 1 & d & c & b & a & 0
\end{array}
\]
Then $\mathbf A=(A,\oplus,\neg,0)$ is an {\rm NMV}-algebra. The operation tables for $\otimes$, $\rightarrow$ and $\sqcup$ look as follows:
\[
\begin{array}{c|cccccc}
\otimes & 0 & a & b & c & d & 1 \\
\hline
  0     & 0 & 0 & 0 & 0 & 0 & 0 \\
	a     & 0 & 0 & 0 & a & 0 & a \\
	b     & 0 & 0 & 0 & 0 & b & b \\
	c     & 0 & a & 0 & a & b & c \\
	d     & 0 & 0 & b & b & a & d \\
	1     & 0 & a & b & c & d & 1
\end{array}
\quad
\begin{array}{c|cccccc}
\rightarrow & 0 & a & b & c & d & 1 \\
\hline
     0      & 1 & 1 & 1 & 1 & 1 & 1 \\
		 a      & d & 1 & d & 1 & 1 & 1 \\
		 b      & c & c & 1 & 1 & 1 & 1 \\
		 c      & b & c & d & 1 & d & 1 \\
		 d      & a & d & c & c & 1 & 1	\\
		 1      & 0 & a & b & c & d & 1
\end{array}
\quad
\begin{array}{c|cccccc}
\sqcup & 0 & a & b & c & e & 1 \\
\hline
  0    & 0 & a & b & c & e & 1 \\
	a    & a & a & c & c & e & 1 \\
  b    & b & c & b & c & e & 1 \\
	c    & c & c & c & c & 1 & 1 \\
	e    & e & e & e & 1 & e & 1 \\
	1    & 1 & 1 & 1 & 1 & 1 & 1
\end{array}
\]
Hence the corresponding poset has the Hasse diagram
\vspace*{-5mm}
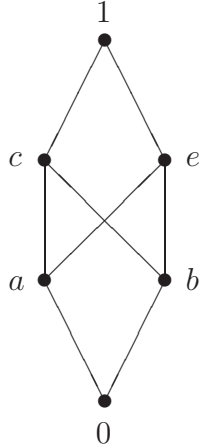
\begin{center}
\setlength{\unitlength}{8mm}
\begin{picture}(4,8)
\put(2,1){\circle*{.2}}
\put(1,3){\circle*{.2}}
\put(3,3){\circle*{.2}}
\put(1,5){\circle*{.2}}
\put(3,5){\circle*{.2}}
\put(2,7){\circle*{.2}}
\put(1,3){\line(1,-2)1}
\put(1,3){\line(1,1)2}
\put(1,3){\line(0,1)2}
\put(3,3){\line(0,1)2}
\put(3,3){\line(-1,1)2}
\put(3,3){\line(-1,-2)1}
\put(2,7){\line(-1,-2)1}
\put(2,7){\line(1,-2)1}
\put(1.85,.3){$0$}
\put(.4,2.85){$a$}
\put(3.35,2.85){$b$}
\put(.4,4.9){$c$}
\put(3.35,4.9){$e$}
\put(1.85,7.3){$1$}
\end{picture}
\end{center}
One can see that this poset is not a lattice, thus $\mathbf A$ cannot be an {\rm MV}-algebra nor a commutative basic algebra. The switching involutions on sections are given by the following table:
\[
\begin{array}{c|cccccc}
 x  & 0 & a & b & c & d & 1 \\
\hline
x^0 & 1 & d & c & b & a & 0 \\
x^a &   & 1 &   & c & d & a \\
x^b &   &   & 1 & d & c & b \\
x^c &   &   &   & 1 &   & c \\
x^d &   &   &   &   & 1 & d \\
x^1 &   &   &   &   &   & 1
\end{array}
\]
All these involutions are antitone which is not the case in general.
\end{Example}

If $\mathbf A=(A,\oplus,\neg,0)$ is an NMV-algebra whose switching section involutions are antitone, then we will call $\mathbf A$ an {\em{\rm NMV}-algebra with {\rm SAI}}.

The next concept which will be used is as follows. Let $\mathbf P=(P,\leq,\otimes,\rightarrow,0,1)$ be a bounded poset with two additional binary operations such that
\begin{enumerate}
\item[(i)] $(P,\otimes,1)$ is a commutative groupoid with neutral element $1$,
\item[(ii)] For all $x,y,z\in P$, $x\otimes y\leq z$ is equivalent to $x\leq y\rightarrow z$.
\end{enumerate}
Then $\mathbf P$ is the called a {\em residuated poset}. Condition (ii) is called {\em adjointness} (see e.g.\ \cite B). In any residuated poset we put $\neg x:=x\rightarrow0$. The residuated poset $\mathbf P$ is called {\em integral} if $1$ is the greatest element of $(P,\leq)$.

We are going to investigate under which conditions a certain modification of adjointness holds in NMV-algebras with SAI.

\begin{Lemma}\label{lem2}
Let $(A,\oplus,\neg,0)$ be an {\rm NMV}-algebra with {\rm SAI}, let $a,b,c\in A$ and assume $c\leq b$. Then $a\otimes b\leq c$ implies $a\leq b\rightarrow c$.
\end{Lemma}

\begin{proof}
Since $c\leq b$ and $\neg$ is antitone, we have $\neg b\le\neg c$. If $a\otimes b\leq c$ then $\neg(a\sqcup\neg b)^{\neg b}\leq c$ which implies $\neg c\leq(a\sqcup\neg b)^{\neg b}$ and hence
\[
a\leq a\sqcup\neg b\leq(\neg c)^{\neg b}=\neg c\rightarrow\neg b=b\rightarrow c.
\]
\end{proof}

Under some other assumption, the converse implication hold.

\begin{Lemma}\label{lem3}
Let $(A,\oplus,\neg,0)$ be an {\rm NMV}-algebra with {\rm SAI}, let $a,b,c\in A$ and assume $\neg a\leq b$. Then $a\leq b\rightarrow c$ implies $a\otimes b\leq c$.
\end{Lemma} 

\begin{proof}
If $a\leq b\rightarrow c$ then $a\leq\neg c\rightarrow\neg b=(\neg c\sqcup\neg b)^{\neg b}$. Since $\neg a\le b$ this implies  $\neg c\leq\neg c\sqcup\neg b\leq a^{\neg b}$ and hence $a\otimes b=\neg(a\sqcup\neg b)^{\neg b}=\neg a^{\neg b}\leq c$.
\end{proof}

Combining Lemma~\ref{lem2} and \ref{lem3} we obtain

\begin{Corollary}
If $(A,\oplus,\neg,0)$ be an {\rm NMV}-algebra with {\rm SAI} then for all $x,y,z\in A$,
\begin{enumerate}
\item[{\rm(i)}] if $\neg x,z\leq y$ then
\[
x\otimes y\leq z\text{ is equivalent to }x\leq y\rightarrow z,
\]
\item[{\rm(ii)}] $(x\sqcup\neg(y\sqcup z))\otimes(y\sqcup z)\leq z$ if and only if $x\sqcup\neg(y\sqcup z)\leq(y\sqcup z)\rightarrow z=y\rightarrow z$.
\end{enumerate}
\end{Corollary}

In Example~\ref{ex1}, the equivalence of $x\otimes y\leq z$ and $x\leq y\rightarrow z$ holds if and only if $(x,y,z)\neq(c,c,d),(d,d,c)$. Observe that neither $d\leq c$ nor $c\leq d$.

Next we introduce a concept which is weaker than that of a residuated poset. Namely, adjointness will be replaced by the conditions occurring in Lemmata~\ref{lem2} and \ref{lem3}.

\begin{Definition}\label{def2}
A sixtuple $\mathbf P=(P,\leq,\otimes,\rightarrow,0,1)$ with two binary operations $\otimes$ and $\rightarrow$ will be called a {\em conditionally residuated poset} if
\begin{enumerate}
\item[{\rm(a)}] $(P,\leq,0,1)$ is a bounded poset and $x\leq y$ implies $x\rightarrow y=1$,
\item[{\rm(b)}] $(P,\otimes,1)$ is a commutative groupoid with neutral element $1$,
\item[{\rm(c)}] if $x\otimes y\leq z$ and $z\leq y$ then $x\leq y\rightarrow z$, and if $x\leq y\rightarrow z$ and $\neg x\leq y$ then $x\otimes y\leq z$.
\end{enumerate}
Here and in the following $\neg x:=x\rightarrow0$ for all $x\in P$. Condition {\rm(c)} will be called {\em conditional adjointness}. We say that $\mathbf P$ satisfies
\begin{itemize}
\item {\em weak divisibility} if $x\otimes(x\rightarrow y)\leq y$,
\item the {\em contraposition law} if $x\rightarrow y\approx\neg y\rightarrow\neg x$,
\item the {\em double negation law} if $\neg(\neg x)\approx x$,
\item the {\em \L ukasiewicz axiom} if $(x\rightarrow y)\rightarrow y\approx(y\rightarrow x)\rightarrow x$,
\item the {\em compatibility conditions} if $y\leq x\rightarrow y$ and $(((x\rightarrow y)\rightarrow y)\rightarrow y\approx x\rightarrow y$.
\end{itemize}
\end{Definition}

In order to justify the introduced concepts, we state the following

\begin{Theorem}
Let $\mathbf A=(A,\oplus,\neg,0)$ be an {\rm NMV}-algebra with {\rm SAI} and the term operations $\otimes$ and $\rightarrow$ defined above and the induced order $\leq$. Then $(A,\leq,\otimes,\rightarrow,0,1)$ is a conditionally residuated poset satisfying weak divisibility, the contraposition law, the double negation law, the \L ukasiewicz axiom and the compatibility conditions.
\end{Theorem}

\begin{proof}
The proof follows by the identities mentioned in and after Definition~\ref{def1} and Lemmata~\ref{lem2} and \ref{lem3}.
\end{proof}

Our next goal is to show that also the converse assertion holds. For this, we have to prove some preliminary results.

\begin{Lemma}\label{lem1}
Every conditionally residuated poset $\mathbf P=(P,\leq,\otimes,\rightarrow,0,1)$ satisfies the following conditions:
\begin{enumerate}
\item[{\rm(i)}] $1\rightarrow x\approx x$, $x\rightarrow x\approx1$ and $\neg0\approx1$,
\item[{\rm(ii)}] If $\mathbf P$ satisfies the double negation law then $\neg x\otimes x\approx0$ and $\neg1\approx0$,
\item[{\rm(iii)}] If $\mathbf P$ satisfies weak divisibility and the compatibility conditions then $x\leq(x\rightarrow y)\rightarrow y$; moreover, $x\leq y$ if and only if $x\rightarrow y=1$ in this case.
\end{enumerate}
\end{Lemma}

\begin{proof}
\
\begin{enumerate}
\item[(i)] Since $1\rightarrow x\leq1\rightarrow x$ and $\neg(1\rightarrow x)\leq1$ we have $1\rightarrow x=(1\rightarrow x)\otimes1\leq x$. On the other hand, from $x\otimes1\leq x$ and $x\leq1$ we obtain $x\leq1\rightarrow x$. Together we have $1\rightarrow x\approx x$. Moreover, since $1\otimes x\leq x$ and $x\leq x$ we have $1\leq x\rightarrow x$, i.e.\ $x\rightarrow x\approx1$. Because of $x\rightarrow x\approx1$ we have $\neg0\approx0\rightarrow0\approx1$.
\item[(ii)] Since $\neg x\leq x\rightarrow0$ and $\mathbf P$ satisfies the double negation law we obtain $\neg x\otimes x\leq0$, i.e.\ $\neg x\otimes x\approx0$ and $\neg1\approx\neg1\otimes1\approx0$.
\item[(iii)] Because of weak divisibility we have $x\otimes(x\rightarrow y)\leq y$. Now $y\leq x\rightarrow y$ according to the compatibility conditions. Hence $x\leq(x\rightarrow y)\rightarrow y$ follows because of (c) of Definition~\ref{def2}. Moreover, according to (a) of Definition~\ref{def2}, $x\leq y$ implies $x\rightarrow y=1$. Conversely, because of $x\leq(x\rightarrow y)\rightarrow y$ and (i), $x\rightarrow y=1$ implies $x\leq(x\rightarrow y)\rightarrow y=1\rightarrow y=y$.
\end{enumerate}
\end{proof}

Let us note that all the conditions occurring in Lemma~\ref{lem1} are satisfied in every NMV-algebra as mentioned above. By using the previous results, we are now able to show when a conditionally residuated poset can be converted into an NMV-algebra.

\begin{Theorem}
If $\mathbf P=(P,\leq,\otimes,\rightarrow,0,1)$ is a conditionally residuated poset satisfying weak divisibility, the contraposition law, the double negation law, the \L ukasiewicz axiom and the compatibility conditions and if we put $x\oplus y:=\neg x\rightarrow y$ for all $x,y\in P$ then $(P,\oplus,\neg,0)$ is an {\rm NMV}-algebra whose induced order coincides with $\leq$.
\end{Theorem}

\begin{proof}
Let $a,b\in P$. Put $x\sqcup y:=(x\rightarrow y)\rightarrow y$ for all $x,y\in P$. Applying (iii) of Lemma~\ref{lem1} and the \L ukasiewicz axiom we have $a\sqcup b=b\sqcup a\in U(a,b)$. Moreover, if $a\leq b$ then $a\sqcup b=(a\rightarrow b)\rightarrow b=1\rightarrow b=b$ according to (i) of Lemma~\ref{lem1}. Since also $1\sqcup a=(1\rightarrow a)\rightarrow a=a\rightarrow a=1$ according to (i) of Lemma~\ref{lem1}, $(P,\sqcup,1)$ is a directoid with greatest element $1$. Now put $x^a:=x\rightarrow a$ for all $x\in[a,1]$. According to the compatibility conditions, $^a$ is a mapping from $[a,1]$ to itself. Because of (i) of Lemma~\ref{lem1}, $a^a=a\rightarrow a=1$ and $1^a=1\rightarrow a=a$, i.e.\ $^a$ is switching. Moreover,
\[
(x^a)^a=((x\sqcup a)^a)^a=(((x\rightarrow a)\rightarrow a)\rightarrow a)\rightarrow a=(x\rightarrow a)\rightarrow a=x\sqcup a=x
\]
for all $x\in[a,1]$ according to the compatibility conditions, i.e.\ $^a$ is an involution on $[a,1]$. Finally, $\neg x\rightarrow y\approx\neg y\rightarrow\neg(\neg x)\approx\neg y\rightarrow x$ (this condition is denoted by (WE) in Theorem~8 of \cite{CK}) because of the contraposition law and the double negation law. This shows that $(P,\oplus,\neg,0)$ is an NMV-algebra according to Theorem~8 in \cite{CK}. Because of (v) of Lemma~\ref{lem1}, the induced order of $(P,\oplus,\neg,0)$ coincides with $\leq$.
\end{proof}

The question arises how to modify the concept of an NMV-algebra in order to obtain a residuated poset. The main difficulty which prevents to reach adjointness for NMV-algebras is, as mentioned above, the lack of associativity of the operation $\oplus$ and hence also of $\otimes$. However, the situation changes if we suppose the operation $\otimes$ to be monotonous, i.e.\ if
\[
x\leq y\text{ implies }x\otimes z\leq y\otimes z
\]
for all $z$. In what follows, we show that in this case, a poset with $\otimes$ can be organized into a residuated poset provided some reasonable conditions are satisfied.

\begin{Theorem}
Let $\mathbf P=(P,\leq,\otimes,\neg,0,1)$ be a bounded poset with an antitone involution and a commutative binary operation $\otimes$ which is monotonous and satisfies the identity $x\otimes1\approx x$. Define $x\rightarrow y:=\neg(x\otimes\neg y)$ for all $x,y\in P$. If $\mathbf P$ satisfies the condition
\begin{equation}\label{equ1}
x\leq\neg(\neg(x\otimes y)\otimes y)
\end{equation}
then $\mathbf R:=(P,\leq,\otimes,\rightarrow,0,1)$ is an integral residuated poset satisfying the double negation law.
\end{Theorem}

\begin{proof}
Let $a,b,c\in P$. If $a\otimes b\leq c$ then
\[
a\leq\neg(\neg(a\otimes b)\otimes b)\leq\neg(\neg c\otimes b)=b\rightarrow c
\]
due to (\ref{equ1}), monotonicity of $\otimes$ and antitony of $\neg$. Conversely, if $a\leq b\rightarrow c$ then
\[
a\otimes b\leq(b\rightarrow c)\otimes b=\neg(b\otimes\neg c)\otimes b=\neg(\neg(\neg(\neg c\otimes b)\otimes b))\leq\neg(\neg c)=c
\]
because of (\ref{equ1}) and since $\neg$ is an antitone involution. Hence $\mathbf R$ is a residuated poset. Since $\neg$ is an antitone involution on $(P,\leq)$ we have $\neg0=1$. Now $a\rightarrow0=\neg(a\otimes\neg0)=\neg(a\otimes1)=\neg a$ and therefore $\mathbf R$ satisfies the double negation law. Since $1$ is the greatest element of $(P,\leq)$, $\mathbf R$ is integral.
\end{proof}

We can prove also the converse.

\begin{Theorem}
Let $\mathbf R=(R,\leq,\otimes,\rightarrow,0,1)$ be an integral residuated poset satisfying the identity $x\rightarrow y\approx\neg(x\otimes\neg y)$ (here and in the following $\neg x:=x\rightarrow 0$) and the double negation law and assume $\neg$ to be antitone. Then $(R,\leq,\otimes,\neg,0,1)$ is a bounded poset with an antitone involution whose operation $\otimes$ is commutative, monotonous and satisfies the identities $x\otimes1\approx x$ and $x\otimes0\approx0$ as well as condition {\rm(\ref{equ1})} .
\end{Theorem}

\begin{proof} Let $a,b,c\in R$. Of course, we have $a\approx(a\rightarrow0)\rightarrow0\approx\neg(\neg a)$. Since $0\leq a\rightarrow 0$, we infer $a\otimes0=0\otimes a\leq0$ whence $a\otimes0=0$. By definition, $\otimes$ is commutative. If $a\leq b$ then because of $c\rightarrow\neg b\leq c\rightarrow\neg b$ we have $(c\rightarrow\neg b)\otimes c\leq\neg b\leq\neg a$ and hence $c\rightarrow\neg b\leq c\rightarrow\neg a$ wherefrom we conclude $c\otimes a=\neg(c\rightarrow\neg a)\leq\neg(c\rightarrow\neg b)=c\otimes b$, i.e.\ $\otimes$ is monotonous. Finally, because of $\neg(a\otimes b)=b\rightarrow\neg a$ we have $\neg(a\otimes b)\otimes b\leq\neg a$ which implies $a\leq\neg(\neg(a\otimes b)\otimes b)$ proving (\ref{equ1}).
\end{proof}

Authors' addresses:

Ivan Chajda \\
Palack\'y University Olomouc \\
Faculty of Science \\
Department of Algebra and Geometry \\
17.\ listopadu 12 \\
771 46 Olomouc \\
Czech Republic \\
ivan.chajda@upol.cz

Helmut L\"anger \\
TU Wien \\
Faculty of Mathematics and Geoinformation \\
Institute of Discrete Mathematics and Geometry \\
Wiedner Hauptstra\ss e 8-10 \\
1040 Vienna \\
Austria \\
helmut.laenger@tuwien.ac.at

\begin{thebibliography}9
\bibitem B
R.~B\v elohl\'avek, Fuzzy Relational Systems. Foundations and Principles. Springer, New York 2002. ISBN 978-1-4613-5168-9.
\bibitem{CK}
I.~Chajda and J.~K\"uhr, A non-associative generalization of MV-algebras. Math. Slovaca {\bf57} (2007), 301-312.
\bibitem{CL11}
I.~Chajda and H.~L\"anger, Directoids. An Algebraic Approach to Ordered Sets. Heldermann, Lemgo 2011. ISBN 978-3-88538-232-4.
\bibitem{CL}
I.~Chajda and H.~L\"anger, Properties of non-associative MV-algebras. Math.\ Slovaca (to appear).
\bibitem C
C.~C.~Chang, Algebraic analysis of many valued logics. Trans.\ Amer.\ Math.\ Soc.\ {\bf88} (1958), 467-490.
\bibitem{CDM}
R.~L.~O.~Cignoli, I.~M.~L.~D'Ottaviano and D.~Mundici, Algebraic Foundations of Many-valued Reasoning. Kluwer, Dordrecht 2000. ISBN 0-7923-6009-5.
\bibitem{JQ}
J.~Je\v zek and R.~Quackenbush, Directoids: Algebraic models of up-directed sets. Algebra Universalis {\bf27} (1990), 49-69.
\end{thebibliography}
\end{document}